\documentclass{amsart}

\usepackage{amssymb}
\usepackage{amscd}
\usepackage{amsfonts}
\usepackage{version}

\numberwithin{equation}{section}

\theoremstyle{plain}

\newtheorem{theorem}[subsection]{Theorem}

\newtheorem{lemma}[subsection]{Lemma}
\newtheorem{corollary}[subsection]{Corollary}
\newtheorem{conjecture}[subsection]{Conjecture}

\newtheorem{question}[subsection]{Question}

\theoremstyle{definition}

\newtheorem{cheat}[subsection]{Cheat}

\renewcommand{\leq}{\leqslant}
\renewcommand{\geq}{\geqslant}

\newsavebox{\proofbox}
\savebox{\proofbox}{\begin{picture}(7,7)%
  \put(0,0){\framebox(7,7){}}\end{picture}}

\newcommand\E{\mathbb{E}}
\newcommand\Z{\mathbb{Z}}
\newcommand\R{\mathbb{R}}

\newcommand\C{\mathbb{C}}

\newcommand\D{\mathcal{D}}

\newcommand\J{\mathcal{J}}

\def\J{\mathbf{J}}
\newcommand\lo{\operatorname{LO}}
\newcommand\cs{\operatorname{CS}}

\newcommand\Symb{\operatorname{Symb}}

\newcommand\eps{\varepsilon}

\newcommand\id{\operatorname{id}}

\newcommand\GI{{\operatorname{GI}}}

\begin{document}
\title{An inverse theorem for the Gowers $U^{s+1}[N]$-norm}
\author{Ben Green}
\address{Centre for Mathematical Sciences\\
Wilberforce Road\\
Cambridge CB3 0WA\\
England }
\email{b.j.green@dpmms.cam.ac.uk}
\author{Terence Tao}
\address{Department of Mathematics\\
UCLA\\
Los Angeles, CA 90095\\
USA}
\email{tao@math.ucla.edu}
\author{Tamar Ziegler}
\address{Department of Mathematics \\
Technion - Israel Institute of Technology\\
Haifa, Israel 32000}
\email{tamarzr@tx.technion.ac.il}

\begin{abstract} This is an announcement of the proof of the \emph{inverse conjecture for the Gowers $U^{s+1}[N]$-norm} for all $s \geq 3$; this is new for $s \geq 4$, the cases $s = 1,2,3$ having been previously established. More precisely we outline a proof that if $f : [N] \rightarrow [-1,1]$ is a function with $\Vert f \Vert_{U^{s+1}[N]} \geq \delta$ then there is a bounded-complexity $s$-step nilsequence $F(g(n)\Gamma)$ which correlates with $f$, where the bounds on the complexity and correlation depend only on $s$ and $\delta$. From previous results, this conjecture implies the Hardy-Littlewood prime tuples conjecture for any linear system of finite complexity. In particular, one obtains an asymptotic formula for the number of $k$-term arithmetic progressions $p_1 < p_2 < \dots < p_k \leq N$ of primes, for every $k \geq 3$. 
\end{abstract}

\maketitle

\section{Introduction}

This is an announcement and summary of our much longer paper \cite{gtz-inverse}, the purpose of which is to establish the general case of the \emph{Inverse Conjecture for the Gowers norms}, conjectured by the first two authors in \cite[Conjecture 8.3]{green-tao-linearprimes}. If $N$ is a (typically large) positive integer then we write $[N] := \{1,\dots,N\}$. Throughout the paper we write $\D = \{z \in \C: |z| \leq 1\}$. For each integer $s \geq 1$ the inverse conjecture $\GI(s)$, whose statement we recall shortly, describes the structure of functions $f : [N] \rightarrow \D$ whose $(s+1)^{\operatorname{st}}$ Gowers norm $\Vert f \Vert_{U^{s+1}[N]}$ is large. These conjectures together with a good deal of motivation and background to them are discussed in \cite{green-icm,green-tao-u3inverse,green-tao-linearprimes}. The conjectures $\GI(1)$ and $\GI(2)$ have been known for some time, the former being a straightforward application of Fourier analysis, and the latter being the main result of \cite{green-tao-u3inverse} (see also \cite{sam} for the characteristic $2$ analogue).  The case $\GI(3)$ was also recently established by the authors in \cite{u4-inverse}. In this note we announce the  resolution of the remaining cases $\GI(s)$ for $s \geq 3$, in particular reproving the results in \cite{u4-inverse}.
 
We begin by recalling the definition of the Gowers norms. If $G$ is a finite abelian group, $d \geq 1$ is an integer, and $f : G \rightarrow \C$ is a function then we define
\begin{equation}\label{ukdef}
 \Vert f \Vert_{U^{d}(G)} := \left(  \E_{x,h_1,\dots,h_k \in G} \Delta_{h_1} \ldots \Delta_{h_d} f(x)\right)^{1/2^d},
\end{equation}
where $\Delta_h f$ is the multiplicative derivative
$$ \Delta_h f(x) := f(x+h) \overline{f(x)}$$
and $\E_{x \in X} f(x) := \frac{1}{|X|} \sum_{x \in X} f(x)$ denotes the average of a function $f: X \to \C$ on a finite set $X$. Thus for instance we have
\[ 
\Vert f \Vert_{U^2(G)} := \left( \E_{x,h_1,h_2 \in G} f(x) \overline{f(n+h_1) f(n+h_2)} f(n+h_1 + h_2)\right)^{1/4}.
\]
One can show that $U^d(G)$ is indeed a norm on the functions $f: G \to \C$ for any $d \geq 2$, though we will not need this fact here.

In this paper we will be concerned with functions on $[N]$, which is not quite a group. To define the Gowers norms of a function $f : [N] \rightarrow \C$, set $G := \Z/\tilde N\Z$ for some integer $\tilde N \geq 2^d N$, define a function $\tilde f : G \rightarrow \C$ by $\tilde f(x) = f(x)$ for $x = 1,\dots,N$ and $\tilde f(x) = 0$ otherwise and set 
\[ \Vert f \Vert_{U^d[N]} := \Vert \tilde f \Vert_{U^d(G)} / \Vert 1_{[N]} \Vert_{U^d(G)},\] where $1_{[N]}$ is the indicator function of $[N]$.  It is easy to see that this definition is independent of the choice of $\tilde N$. One could take $\tilde N := 2^d N$ for definiteness if desired. 

The \emph{Inverse conjecture for the Gowers $U^{s+1}[N]$-norm}, abbreviated as $\GI(s)$, posits an answer to the following question.

\begin{question}
Suppose that $f : [N] \rightarrow \D$ is a function and let $\delta > 0$ be a positive real number. What can be said if $\Vert f \Vert_{U^{s+1}[N]} \geq \delta$?
\end{question}

Note that in the extreme case $\delta = 1$ one can easily show that $f$ is a phase polynomial, namely  $f(n)=e(P(n))$ for some polynomial $P$ of degree at most $s$.  Furthermore, if $f$ correlates with a phase polynomial, that is to say if $|\E_{n \in [N]} f(n) \overline{e( P(n))}| \geq \delta$, then it is easy to show that
$\Vert f \Vert_{U^{s+1}[N]} \geq c(\delta)$. It is natural to ask whether the converse is also true - does a large Gowers norm imply correlation with a polynomial phase function?
Surprisingly, the answer is no, as was observed by Gowers \cite{gowers-4aps} and, in the related context of \emph{multiple recurrence}, somewhat earlier by Furstenberg and Weiss \cite{furst, fw-char}. The work of Furstenberg and Weiss draws attention to the role of homogeneous spaces $G/\Gamma$ of nilpotent Lie groups, and subsequent work of Host and Kra \cite{host-kra} provides a link, in an ergodic-theoretic context, between these spaces and certain seminorms with a formal similarity to the Gowers norms under discussion here. Later work of Bergelson, Host and Kra \cite{bhk} highlights the role of a class of functions arising from these spaces $G/\Gamma$ called \emph{nilsequences}. The inverse conjecture for the Gowers norms, first formulated precisely in \S 8 of \cite{green-tao-linearprimes}, postulates that this class of functions (which contains the polynomial phases) represents the full set of obstructions to having large Gowers norm. 

Here is that precise formulation of the conjecture. Recall that an \emph{$s$-step nilmanifold} is a manifold of the form $G/\Gamma$, where $G$ is a connected, simply-connected nilpotent Lie group of step at most $s$ (i.e. all $(s+1)$-fold commutators of $G$ are trivial), and $\Gamma$ is a lattice (a discrete cocompact subgroup of $G$).

\begin{conjecture}[$\GI(s)$]\label{gis-conj}  Let $s \geq 0$ be an integer, and let $0 < \delta \leq 1$.   Then there exists a finite collection ${\mathcal M}_{s,\delta}$ of $s$-step nilmanifolds $G/\Gamma$, each equipped with some smooth Riemannian metric $d_{G/\Gamma}$ as well as constants $C(s,\delta), c(s,\delta) > 0$ with the following property. Whenever $N \geq 1$ and $f : [N] \rightarrow \D$ is a function such that $\Vert f \Vert_{U^{s+1}[N]} \geq \delta$, there exists a nilmanifold $G/\Gamma \in {\mathcal M}_{s,\delta}$, some $g \in G$ and a function $F: G/\Gamma \to \D$ with Lipschitz constant at most $C(s,\delta)$ with respect to the metric $d_{G/\Gamma}$, such that
$$ |\E_{n \in [N]} f(n) \overline{F(g^n x)}| \geq c(s,\delta).$$
\end{conjecture}

Let us briefly review the known partial results on this conjecture (in no particular order):  
\begin{enumerate}
\item $\GI(0)$ is trivial.
\item $\GI(1)$ follows from a short Fourier-analytic computation.
\item $\GI(2)$ was established five years ago in \cite{green-tao-u3inverse}, building on work of Gowers \cite{gowers-4aps}.
\item $\GI(3)$ was established, quite recently, in \cite{u4-inverse}.
\item In the extreme case $\delta = 1$ one can easily show that $f(n)=e(P(n))$ for some polynomial $P$ of degree at most $s$, and every such function \emph{is} an $s$-step nilsequence by a direct construction. See, for example, \cite{green-tao-u3inverse} for the case $s = 2$. 
\item In the almost extremal case $\delta \geq 1- \eps_s$, for some $\eps_s > 0$, one may see that $f$ correlates with a phase $e(P(n))$ by adapting arguments first used in the theoretical computer-science literature \cite{akklr}.
\item The analogue of $\GI(s)$ in ergodic theory (which, roughly speaking, corresponds to the asymptotic limit $N \to \infty$ of the theory here; see \cite{host-kra-uniformity} for further discussion) was formulated and established in \cite{host-kra}, work done independently of the work of Gowers (see also the earlier paper \cite{hk1}). This work was the first place in the literature to link objects of Gowers-norm type (associated to functions on a measure-preserving system $(X, T,\mu)$) with flows on nilmanifolds, and the subsequent paper \cite{bhk} was  the first work to underline the importance of \emph{nilsequences}. The formulation of $\GI(s)$ by the first two authors in \cite{green-tao-linearprimes} was very strongly influenced by these works. For the closely related problem of analysing multiple ergodic averages, the relevance of flows on nilmanifolds was earlier pointed out in  \cite{furst, fw-char,lesigne-nil}, building upon earlier work in  \cite{conze}.  See also \cite{hk0,ziegler} for related work on multiple averages and nilmanifolds in ergodic theory.
\item The analogue of $\GI(s)$ in finite fields of large characteristic was established by ergodic-theoretic methods in \cite{bergelson-tao-ziegler,tao-ziegler}.
\item A weaker ``local'' version of the inverse theorem (in which correlation takes place on a subprogression of $[N]$ of size $\sim N^{c_s}$) was established by Gowers \cite{gowers-longaps}. This paper provided a good deal of inspiration for our work here.
\item The converse statement to $\GI(s)$, namely that correlation with a function of the form $n \mapsto F(g^n x)$ implies that $f$ has large $U^{s+1}[N]$-norm, is also known. This was first established in \cite[Proposition 12.6]{green-tao-u3inverse}, following arguments of Host and Kra \cite{host-kra} rather closely. A rather simple proof of this result is given in \cite[Appendix G]{u4-inverse}. 
\end{enumerate}

The aim of this announcement is to outline an argument for the general case. Details may be found in the much longer paper \cite{gtz-inverse}.

\begin{theorem}\label{mainthm}  For any $s \geq 3$,
The inverse conjecture for the $U^{s+1}[N]$-norm, $\GI(s)$, is true.
\end{theorem}

By combining this result with the previous results in \cite{green-tao-linearprimes,green-tao-mobiusnilsequences} we obtain a quantitative Hardy-Littlewood prime tuples conjecture for all linear systems of finite complexity; in particular, we now have the expected asymptotic for the number of primes $p_1 < \ldots < p_k \leq N$ in arithmetic progression, for every fixed positive integer $k$.  We refer to \cite{green-tao-linearprimes} for further discussion, as we have nothing new to add here regarding these applications.  Several further applications of the $\GI(s)$ conjectures are given in \cite{fhk,green-tao-arithmetic-regularity}. 

We remark that an alternative strategy towards the inverse conjecture and related problems is currently being developed by Balazs Szegedy in an ongoing series of papers \cite{szeg-1,szeg-2,szeg-3}.  There are some similarities in method between these papers and ours, the most obvious being a reliance on nonstandard analysis to make the algebraic manipulations easier. In other respects the methods of Szegedy are closer to the ergodic theory methods of Host and Kra \cite{host-kra}, whereas ours are ultimately based on the Fourier-analytic methods of Gowers \cite{gowers-4aps,gowers-longaps}.  

In order to avoid some notational and technical difficulties, the presentation in this announcement will be non-rigorous, focusing on various model special cases and ignoring some fine distinctions.  We will indicate these non-rigorous simplifications throughout this paper as ``cheats''.\vspace{8pt}

\emph{Acknowledgements. } BG was, for some of the period during which this work was carried out, a fellow of the Radcliffe Institute at Harvard. He is very grateful to the Radcliffe Institute for providing excellent working conditions. TT is supported by NSF Research Award DMS-0649473, the NSF Waterman award and a grant from the MacArthur Foundation. TZ is supported by ISF grant  557/08, an Alon fellowship and a Landau fellowship of the Taub foundation. All three authors are very grateful to the University of Verona for allowing them to use classrooms at Canazei during a week in July 2009. This work was largely completed during that week.

\section{Reduction to an integration problem}

Our proof of $\GI(s)$ follows the strategy used to establish the $s=2$ case in \cite{green-tao-u3inverse} and the $s=3$ case in \cite{u4-inverse}, these methods in turn being based on the earlier arguments of Gowers \cite{gowers-4aps,gowers-longaps}.  In each case, one uses $\GI(s-1)$ as an induction hypothesis to assist in proving $\GI(s)$.  To pass from $\GI(s-1)$ to $\GI(s)$, one has to perform a ``cohomological'' task, namely that of showing that a certain ``cocycle'' is essentially a ``coboundary'' (or showing that a certain ``closed'' form is essentially ``exact'').  This cohomological task is by far the most difficult portion of the argument, and will be discussed in more detail in later sections. We focus for now on the reduction to that goal.

\begin{cheat}
It will be convenient to suppress dependence on parameters such as $\delta$, and instead use asymptotic notation such as $\ll$ or $O(1)$ liberally.  In the full paper \cite{gtz-inverse}, we will in fact use the language of nonstandard analysis to systematically suppress all of these parameters and make asymptotic notation such as this rigorous. Here, however,  we will avoid the use of this language and instead rely on more informal terminology such as ``bounded'' or ``large''.
\end{cheat}

Fix a positive integer $s \geq 3$ and assume $\GI(s-1)$ as an induction hypothesis. Our goal, of course, is to prove $\GI(s)$.  Suppose then that we have a function $f: [N] \to \D$ with $\|f\|_{U^{s+1}[N]} \gg 1$; our aim is to show that $f$ \emph{correlates} with some \emph{nilsequence} $\chi(n)$ of step $s$ in the sense that
$$ |\E_{n \in [N]} f(n) \overline{\chi(n)}| \gg 1.$$
Here $\chi(n)$ is a function of the form $F(g^n x)$, where $F$ is a Lipschitz function with bounded Lipschitz norm on an $s$-step nilmanifold $G/\Gamma$ chosen from a bounded list of possibilities, $g \in G$, and $x \in G/\Gamma$.  A simple example of an $s$-step nilsequence to keep in mind for now is $\chi(n) = e(\alpha n^s)$, where $\alpha \in \R/\Z$ and $e(x) := e^{2\pi ix}$ is the standard character. We caution however that this is not an especially representative example. Further examples will be discussed later on.

Using the identity
\begin{equation}\label{fasn}
 \| f\|_{U^{s+1}(\Z/\tilde N \Z)} = (\E_{h \in \Z/\tilde N \Z} \| \Delta_h f \|_{U^s(\Z/\tilde N\Z)}^{2^s})^{1/2^{s+1}},
\end{equation}
(extending $f$ by zero outside of $[N]$) it is a simple matter to conclude that
$$ \| \Delta_h f \|_{U^s[N]} \gg 1$$
for \emph{many} $h \in [-N,N]$, by which we mean for all $h$ in a subset $H \subseteq [-N,N]$ with $|H| \gg N$.  Applying the hypothesis $\GI(s-1)$, we conclude that for many $h \in [-N,N]$, there exists a nilsequence $\chi_h$ of step $s-1$ which correlates with $\Delta_h f$, that is to say
\begin{equation}\label{enhf}
 |\E_{n \in [N]} \Delta_h f(n) \overline{\chi_h(n)}| \gg 1.
\end{equation}

Our goal is to show that $f$ correlates with an $s$-step nilsequence $\theta$. Heuristically, then, we expect the $(s-1)$-step nilsequences $\chi_h$ to behave like a derivative $\Delta_h \theta$ of such a nilsequence.  Suppose that we are in a situation where the $\chi_h$ do indeed ``behave like'' $\Delta_h \theta$ in an ostensibly rather weak way, namely
\begin{equation}\label{chai}
 \chi_h = \Delta_h \theta \cdot \psi_h
\end{equation}
where the $\psi_h$ are ``lower-order'' $(s-2)$-step nilsequences. Then we can rewrite \eqref{enhf} as
$$
 |\E_{n \in [N]} \Delta_h(f\overline{\theta})(n) \overline{\psi_h(n)}| \gg 1.
$$
Using the converse to $\GI(s-2)$ (see e.g. \cite[Appendix G]{u4-inverse}), we conclude that
$$
 \|\Delta_h(f\overline{\theta})\|_{U^{s-1}[N]} \gg 1
$$
for many $h \in [-N,N]$.  Using \eqref{fasn} (with $s-1$ in place of $s$), we conclude that
$$
 \| f\overline{\theta} \|_{U^{s}[N]} \gg 1.
$$
By a further appeal to the inductive hypothesis $\GI(s-1)$, we have
$$ |\E_{n \in [N]} f(n) \overline{\theta(n)} \overline{\psi(n)}| \gg 1$$
for some $(s-1)$-step nilsequence $\psi$.  Since $\theta \psi$ is an $s$-step nilsequence, we obtain the claim.

We may thus formulate our ``cohomological'' task more precisely: we must show that $h \mapsto \chi_h$ is a ``coboundary'' in the sense that \eqref{chai} holds for many $h$, and some $s$-step nilsequence $\theta$ and $(s-1)$-step nilsequences $\psi_h$.  

\begin{cheat} Actually, this is an oversimplification in a number of minor ways.  For instance, it is convenient to allow the two factors of $\theta$ that appear in $\Delta_h \theta(n) = \theta(n+h) \overline{\theta(n)}$\eqref{chai} to be distinct. In other words, we have a representation \begin{equation}\label{chia-weak}
 \chi_h = \theta(n+h) \overline{\theta'(n)} \psi_h(n)
\end{equation}
for some nilsequences $\theta,\theta'$ of degree $s$. The above arguments can be adapted to this case by using the \emph{Cauchy-Schwarz-Gowers inequality} (see \cite{gowers-longaps}) to decouple $\theta$ and $\theta'$.  Secondly, for technical reasons having to do with a topological obstruction that we will discuss in the next section, the nilsequences here will be vector- rather than scalar-valued. Finally, in the actual proof, one needs to modify $\chi_h$ at various stages of the argument to a slightly different nilsequence $\chi'_h$ which still correlates with $\Delta_h f$, and so \eqref{chai} would apply to the nilsequences $\chi'_h$ rather than $\chi_h$.

To keep the exposition simple (at the expense of strict accuracy), we will ignore these details and pretend that our goal is to establish a representation of the form \eqref{chai}.
\end{cheat}


\section{Nilcharacters}
\label{sec3}
Our arguments are geared towards the case $s \geq 3$, but let us temporarily consider the $s=2$ case as motivation.  In that case, the $\chi_h$ are $1$-step nilsequences.  It is not difficult to see that such sequences take the form $\chi_h(n) = F( \xi_h n )$ where $F: (\R/\Z)^d \to \C$ is a Lipschitz function on a torus of bounded dimension $d=O(1)$, and $\xi_h \in \R^d$ is a vector-valued frequency.  These sequences were obtained from the hypothesis $\GI(1)$, which asserts that functions of large $U^2$-norm correlate with a $1$-step nilsequence.

The space of $1$-step nilsequences is in some sense ``generated'' by a special type of $1$-step nilsequence, namely the \emph{Fourier characters} $n \mapsto e( \xi n )$ where $\xi \in \R$ is some \emph{frequency}.  Indeed, from Fourier analysis or the Stone-Weierstrass theorem it is easy to see that every $1$-step nilsequence can be approximated uniformly to arbitrary accuracy by a bounded linear combination of Fourier characters.  In particular, $\GI(1)$ implies that functions of large $U^2$-norm correlate with a Fourier character.

Fourier characters have several additional pleasant properties inside the space of $1$-step nilsequences.  For instance, we have the following facts.
\begin{enumerate}
\item They always have magnitude $1$, and can therefore be inverted by their conjugate: $e( \xi n ) \overline{e(\xi n)} = 1$.
\item They form an abelian group under multiplication.  \item They are translation-invariant modulo lower order terms: for any $h$, $e(\xi(n+h))$ and $e(\xi n)$ differ only by a constant depending on $h$ (i.e. a $0$-step nilsequence).  
\item The mean $\E_{n \in [N]} e(\xi n)$ of a Fourier character is negligible unless the frequency $\xi$ is extremely small (more precisely, if $\xi = O(1/N)$), in which case the character $e(\xi n)$ is ``essentially constant'' (and thus essentially a $0$-step nilsequence).
\end{enumerate}

For the more general argument, as in many other places \cite{green-tao-nilratner,u4-inverse,host-kra}, it is convenient to define a notion of \emph{nilcharacter} in such a way that analogues of the above four properties are still satisfied.  

For $s \geq 2$, an $s$-step nilmanifold $G/\Gamma$ is usually not a torus; however, it is always a torus \emph{bundle} over an $(s-1)$-step nilmanifold $G/G_{s} \Gamma$ with structure group equal to the torus $G_{s}/\Gamma_{s}$, where $G = G_0 = G_1 \supseteq G_2 \supseteq \ldots \supseteq G_{s} \geq \{\id\}$ is the lower central series of the $s$-step nilpotent group $G$, and $\Gamma_i := \Gamma \cap G_i$.  An $s$-step \emph{nilcharacter} is then an $s$-step nilsequence $n \mapsto F( g^n x )$, where $g \in G$, $x \in G/\Gamma$, and $F$ is a Lipschitz function with $|F|=1$ pointwise and obeying the \emph{vertical frequency} condition
\begin{equation}\label{fgx}
F( g_{s} x ) = e( \xi(g_{s}) ) F(x)
\end{equation}
for all $x \in G/\Gamma$, $g_s \in G_s$, and for some continuous homomorphism $\xi: G_s \to \R/\Z$ that annihilates $\Gamma_s$. The homomorphism $\xi$ is referred to as the \emph{vertical frequency} of the nilcharacter.  

\begin{cheat}
For technical reasons it is convenient to generalise the concept of a nilcharacter by replacing the lower central series $G = G_0 = G_1 \supseteq G_2 \supseteq \ldots$ by a more general filtration $G = G_{(0)} \supseteq G_{(1)} \supseteq G_{(2)} \supseteq \ldots$ obeying the inclusion $[G_{(i)},G_{(j)}] \subseteq G_{(i+j)}$, and replacing the linear sequence $n \mapsto g^n x$ by a more general \emph{polynomial sequence} $n \mapsto g(n)$ adapted to this filtration. This generalisation is needed in order to obtain a clean quantitative equidistribution theory for nilsequences and nilcharacters, as explained in some detail in \cite{green-tao-nilratner}.  We will however gloss over the distinction between linear sequences on nilmanifolds and polynomial sequences on filtered nilmanifolds here.

\end{cheat}

A basic example of an $s$-step nilcharacter is a polynomial phase $n \mapsto e(P(n))$, where $P: \Z \to \R/\Z$ is a polynomial of degree at most $s$.  An important family of \emph{near}-examples of nilcharacters come from the more general class of \emph{bracket polynomial phases}, of which the bracket quadratic phase $n \mapsto e( \alpha n \lfloor \beta n \rfloor )$ for some $\alpha,\beta \in \R$ (with $\lfloor \cdot \rfloor$ being the greatest integer function) is a simple model example.  This sequence can \emph{almost} be expressed as a $2$-step nilcharacter on the Heisenberg nilmanifold, which is often presented using $3 \times 3$ matrices (see e.g. \cite{green-tao-u3inverse,u4-inverse}). Here we present the same construction slightly more abstractly, since this will be helpful later. 

Consider, then,  the free $2$-step nilpotent Lie group $G$ generated by elements $e_1, e_2$ such that all commutators of order $3$ or higher, such as $[e_1, [e_1, e_2]]$, are trivial. Here, as is fairly standard in group theory, we write $[x,y] = x^{-1} y^{-1} x y$. A typical element of $G$ has the form $(t_1,t_2,t_{12}) := e_1^{t_1} e_2^{t_2} [e_1, e_2]^{t_{12}}$, $t_1, t_2, t_{12} \in \R$, and multiplication in these coordinates is given by
\[ (t_1, t_2, t_{12}) \ast (t'_1, t'_2, t'_{12}) = (t_1 + t'_1, t_2 + t'_2, t_{12} + t'_{12} + t'_1 t_2).\]
In particular we may identify the discrete subgroup $\Gamma$ consisting of those elements with integer coordinates. Then $G/\Gamma$ is a nilmanifold and a given point with coordinates $(t_1, t_2, t_{12})$ is equivalent under the right action of $\Gamma$ to the point
\[ (\{t_1\}, \{t_2\}, \{t_{12} - \lfloor t_2\rfloor t_1\}).\] This identifies those points  of $G$ with coordinates satisfying $0 \leq t_1, t_2 , t_{12} \leq 1$ as a fundamental domain for the right action of $\Gamma$ on $G$.

One can easily calculate, for specific $g , x\in G$, coordinates for $g^n x$ in the fundamental domain for $G/\Gamma$. In so doing one already sees objects such as $\alpha n \lfloor \beta n \rfloor$ making an appearance. These calculations are even easier if, instead, we look at $g(n)\Gamma$ with $g(n) := e_1^{\alpha n} e_2^{\beta n}$, this being a example of a polynomial sequence on the Heisenberg group $G$ (adapted to the lower central series filtration on $G$). We have
\[ g(n)\Gamma = (\{\alpha n\}, \{\beta n\}, \{ - \lfloor \beta n \rfloor \alpha n\})\Gamma.\] In particular we see that 
\[ F(g(n)\Gamma) = e(\alpha n \lfloor \beta n \rfloor),\]
 where $F : G/\Gamma \rightarrow \C$ is the function defined by 
 \[ F((x,y,z)\Gamma) := e(-z)\] when $0 \leq  x, y, z < 1$.

Why, then, is this a \emph{near} example of a nilsequence and not an actual example? The answer lies in the function $F$, which is unfortunately discontinuous at the edges of the fundamental domain.  This is inevitable due to the twisted nature of the torus bundle that forms the Heisenberg nilmanifold.  However, if one allows nilsequences to be vector-valued instead of scalar-valued, one can avoid this topological obstruction.  For instance, if $1 = \eta_1(x)^2 + \eta_2(x)^2$ is a partition of unity on $\R/\Z$ with $\eta_1, \eta_2$ supported in $[0.1,0.9]$ and $[-0.4,0.4]$ (say) respectively, then the vector-valued sequence 
\begin{equation}\label{vector}
n \mapsto ( e(\alpha n \lfloor \beta n \rfloor) \eta_1(\beta n \hbox{ mod } 1), e(\alpha n \lfloor \beta n - \frac{1}{2} \rfloor) \eta_2(\beta n \hbox{ mod } 1)  )
\end{equation}
will be a nilsequence (taking values in the unit sphere $S^3$ of $\C^2$) associated to the Heisenberg nilmanifold; the piecewise discontinuities of the greatest integer part function have been avoided by use of the cutoffs $\eta_1, \eta_2$, making the relevant function $F$ genuinely Lipschitz and not merely piecewise Lipschitz.

\begin{cheat}  As one can see, the vector-valued nilcharacters such as \eqref{vector} are more complicated than their scalar almost-nilcharacter counterparts such as $e( \alpha n \lfloor \beta n \rfloor )$.  To avoid some distracting notational complications, we will cheat by pretending that sequences such as $e(\alpha n \lfloor \beta n \rfloor )$ are genuine nilcharacters.  With this cheat, we can pretend that all nilsequences involved are scalar-valued rather than vector-valued, and we can use bracket polynomial phases as motivating examples of nilsequences.  For instance, with this cheat, $e( \alpha n^2 )$ and $e( \alpha n \lfloor \beta n \rfloor )$ are $2$-step nilcharacters,
$$e( \alpha n^3 ), e( \alpha n \lfloor \beta n^2 \rfloor ), e( \alpha n^2 \lfloor \beta n \rfloor ), e( \alpha n \lfloor \lfloor \beta n \rfloor \gamma n \rfloor ), e( \alpha n \lfloor \beta n \rfloor \lfloor \gamma n \rfloor )$$
are $3$-step nilcharacters (for $\alpha,\beta,\gamma \in \R$), and so forth.  Indeed, there is a sense in which bracket polynomial phases are essentially the \emph{only} examples of nilcharacters; see \cite{leibman} for further discussion (and \cite{green-tao-u3inverse} for a discussion of the $2$-step case).
\end{cheat}

Nilcharacters enjoy analogues of the four useful properties mentioned earlier:
\begin{enumerate}
\item They have magnitude $1$ (and are thus essentially inverted by their complex conjugations, if this statement is interpreted suitably in the vector-valued case).
\item They (essentially) form an abelian group under multiplication (again using a suitable interpretation of this statement in the vector-valued case, using tensor products).
\item They are essentially translation-invariant modulo lower step errors, much as a polynomial $P(n)$ of degree $s$ is translation-invariant modulo degree $(s-1)$ errors.  In particular, the derivative $\Delta_h \theta$ of an $s$-step nilcharacter is an $(s-1)$-step nilsequence.  
\item The mean $\E_{n \in [N]} \chi(n)$ of a nilcharacter is negligible unless $\chi$ can be represented as an $(s-1)$-step nilsequence.  This property is a consequence of the quantitative equidistribution theory of nilsequences \cite{green-tao-nilratner}.
\end{enumerate}

By using Fourier analysis or the Stone-Weierstrass theorem much as in the $1$-step case, one can show that any $(s-1)$-step nilsequence can be approximated uniformly to arbitrary accuracy by a bounded linear combination of $(s-1)$-step nilcharacters.  Because of this, we can assume without loss of generality that the $\chi_h$ in \eqref{enhf} are $(s-1)$-step nilcharacters, rather than merely $(s-1)$-step nilsequences. That is, we assume henceforth that
\begin{equation}\label{enhf-char}
|\E_{n \in [N]} \Delta_h f(n) \overline{\chi_h(n)}| \gg 1
\end{equation}
for many $h \in [-N,N]$, where the $\chi_h$ are $(s-1)$-step nilcharacters.\vspace{8pt}

\emph{Remarks.} The space of $(s-1)$-step nilcharacters, modulo $(s-2)$-step errors, is denoted $\Symb_{s-1}(\Z)$ in \cite{gtz-inverse}; thus for instance $\Symb_1(\Z) = \hat \Z \equiv \R/\Z$ is the Pontryagin dual of $\Z$, and the abelian group $\Symb_{s-1}(\Z)$ for higher $s$ can be viewed as a higher order generalisation of the Pontryagin dual. As hinted in the above discussion, there is a close relationship between nilcharacters and bracket polynomial phases; some aspects of this relationship are explored in \cite{leibman}. These two types of object can be viewed as two different perspectives on the same concept, with the nilcharacter perspective being superior for understanding equidistribution, and the bracket polynomial perspective being superior for direct (albeit messy) algebraic manipulation.  (The equidistribution of bracket polynomials is studied directly in \cite{ha2}, but it seems cleaner to study equidistribution via nilcharacters instead.) Furthermore, bracket polynomials are a useful source of examples for building intuition.  

In an early version of the full paper \cite{gtz-inverse}, the theory of both nilcharacters and bracket polynomials, together with the connections between them, was extensively developed.  Unfortunately this led to a significant increase in the length of the paper.  The current version of the paper discards the theory of bracket polynomials, and works purely through the formalism of nilcharacters. This has shortened and simplified the paper considerably, albeit at the cost of making some of the algebraic manipulations more abstract. In this announcement, we will rely on bracket polynomial examples for motivation. However, we will indicate at various junctures how various concepts concerning bracket polynomials may be translated into the nilcharacter framework.

\section{Approximate linearity}\label{approx-lin-section}

We return to the problem of establishing a representation of $\chi_h$ that is roughly of the form \eqref{chai}, that is to say
\begin{equation}\label{chai-repeat} \chi_h = \Delta_h \theta \cdot \psi_h \end{equation} for some $s$-step nilsequence $\theta$ and some $(s-2)$-step ``errors'' $\psi_h$. As a consequence of the discussion in the preceding section, we may assume that each $\chi_h$ is a nilcharacter.  

Suppose for the moment that $\chi_h(n)$ was in fact exactly equal to $\Delta_h \theta(n)$ for some $s$-step nilcharacter $\theta$ for all $n, h \in \Z$. Then $\chi_h$ would necessarily obey the \emph{cocycle equation}
\begin{equation}\label{cocycle}
 \chi_{h+k}(n) = \chi_h(n) \chi_k(n+h)
\end{equation}
for all $n,h,k \in \Z$.  

In the converse direction, the cocycle equation \eqref{cocycle} is a sufficient condition to have a representation of the form $\chi_h = \Delta_h \theta$ for \emph{some} function $\theta: \Z \to S^1$ (not necessarily a nilcharacter). Indeed, one can simply set $\theta(n) := \chi_n(0)$, since \eqref{cocycle} then gives
$$ \theta(n+h) = \theta(n) \chi_h(n)$$
for all $n, h$.  To put it another way, when one works in the category of all unit magnitude sequences, rather than the category of nilcharacters, the first cohomology group $H^1(\Z,S^1)$ of the integers is trivial.  

These observations then suggest a strategy for obtaining the desired representation \eqref{chai-repeat} for the $(s-1)$-step nilcharacters $\chi_h$.  One would first show that the nilcharacters $\chi_h$ obey some property resembling the cocycle equation \eqref{cocycle}; then, one would use that cocycle equation, together with the triviality of some sort of ``first cohomology group'' of the integers, to ``integrate'' the cocycle and obtain \eqref{chai-repeat}.

We begin with the first stage. The cocycle property \eqref{cocycle} was, of course deduced from the assumption that $\chi_h = \Delta_h \theta$ exactly. We, however, are operating under the much weaker assumption that $\chi_h$ merely \emph{correlates} with $\Delta_h \theta$, for many $h$, up to lower order terms. To handle this we use an application of the Cauchy-Schwarz inequality due to Gowers \cite{gowers-4aps}. The conclusion of this is as follows.

\begin{lemma}[Approximate cocycle equation]\label{gow}  Suppose that $f: [N] \to \D$ is a function, and that for all $h$ in a dense subset $H \subseteq [-N,N]$ the derivative $\Delta_h f$ correlates with $\chi_h$ for some function $\chi_h: \Z \to \D$.  Then for $\gg N^3$ additive quadruples $h_1,h_2,h_3,h_4 \in H$ \textup{(}that is, quadruples with $h_1 + h_2 = h_3 + h_4$\textup{)} one has
\begin{equation}\label{xanch}
 |\E_{n \in [N]}  \chi_{h_1}(n) \chi_{h_2}(n+h_1-h_4) \overline{\chi_{h_3}(n) \chi_{h_4}(n+h_1-h_4)}| \gg 1.
 \end{equation}
\end{lemma}
\begin{proof} We may clearly replace $\chi_h(n)$ by $e(\theta_h) \chi_h(n)$, for any phases $\theta_h \in \R$. Choose the $\theta_h$ in such a way that, once this replacement is made, $\E_n \Delta_h f(n) \chi_h(n)$ is real and positive. Taking expected values over $h$ and making the substitution $m := n + h$ gives 
\[ \E_{m,n} f(m) \overline{f(n)} \chi_{m - n}(n) \gg 1.\]
Now apply the Cauchy-Schwarz inequality in the variables $m,n$ in turn to eliminate the bounded quantities $f(m)$ and $f(n)$, obtaining
\[ \E_{m, m', n,n'} \chi_{m-n}(n) \overline{\chi_{m' - n}(n)\chi_{m - n'}(n') }\chi_{m' - n'}(n') \gg 1.\]
This is equivalent to the stated result, as one may see upon substituting $m - n = h_1$, $m - n' = h_4$, $m' - n = h_3$ and $m' - n' = h_2$.\end{proof}

\emph{Remarks.} To relate the above lemma to the preceding discussion of cocycle equations, suppose that \eqref{cocycle} is always satisfied. Then one may easily prove that 
\begin{equation}\label{exact-gowers-cocycle} \chi_{h_1}(n) \chi_{h_2}(n+h_1-h_4) \overline{\chi_{h_3}(n) \chi_{h_4}(n+h_1-h_4)} = 1\end{equation} identically whenever $h_1+h_2=h_3+h_4$, a statement which obviously bears comparison to \eqref{xanch}.  Indeed, from \eqref{cocycle} one has
\[ \chi_{h_1}(n) = \chi_{h_3}(n) \chi_{h_1-h_3}(n+h_3)\] and  \[ \chi_{h_4}(n+h_1-h_4) = \chi_{h_2}(n+h_1-h_4) \chi_{h_4-h_2}(n+h_1+h_2-h_4),\]
while from the additive quadruple property one has \[ \chi_{h_4-h_2}(n+h_1+h_2-h_4) = \chi_{h_1-h_3}(n+h_3).\]
Putting these together confirms \eqref{exact-gowers-cocycle}. It is perhaps interesting to note that little has been lost in passing from \eqref{cocycle} to \eqref{exact-gowers-cocycle} (and so we may be confident that little has been lost in asserting Lemma \ref{gow}). Indeed, if \eqref{exact-gowers-cocycle} holds then applying it with $(h_1,h_2,h_3,h_4)=(h+k,0,h,k)$ gives
$$ \chi_{h+k}(n) \chi_0(n+h) = \chi_h(n) \chi_k(n+h).$$
This is almost \eqref{cocycle}.  Setting $\theta(n) := \chi_n(0)$ and $\theta'(n) := \chi_n(0) \overline{\chi_0(n)}$ then gives
\begin{equation}\label{chacha}
 \chi_h(n) = \theta(n+h) \overline{\theta'(n)},
\end{equation}
which is a variant of \eqref{chia-weak}.  Conversely, it is easy to verify that any $\chi_h$ of the form \eqref{chacha} (with $\theta, \theta'$ having magnitude $1$) obeys \eqref{exact-gowers-cocycle}.  This helps explain why our arguments will end up concluding \eqref{chia-weak} rather than \eqref{chai-repeat}.

From the properties of nilcharacters mentioned in the previous section, an immediate corollary of Lemma \ref{gow} is the following.

\begin{corollary}[Top order approximate linearity]\label{approxa}  Let $f: [N] \to \D$ be a function, and suppose that for all $h$ in a dense subset $H \subseteq [-N,N]$ the derivative $\Delta_h f$ correlates with an $(s-1)$-step nilcharacter $\chi_h$.  Then for many additive quadruples $h_1,h_2,h_3,h_4 \in H$ the $(s-1)$-step nilcharacter $\chi_{h_1} \chi_{h_2} \overline{\chi_{h_3}} \overline{\chi_{h_4}}$ is an $(s-2)$-step nilsequence.
\end{corollary}

This corollary asserts that the map $h \mapsto \chi_h$ is in some sense approximately (affine-)linear to top order.  Because it only controls the top order behaviour of $\chi_h$, this corollary is strictly weaker than Lemma \ref{gow}, and will turn out to be insufficient by itself for the purposes of integrating $\chi_h$ in the sense of \eqref{chai-repeat} (or \eqref{chia-weak}). Eventually we need to return to Lemma \ref{gow} and study the \emph{lower-order} (and more specifically, the $(s-2)$-step) terms in more detail.  Nevertheless, Corollary \ref{approxa} is an important partial result and it yields a crucial \emph{linearisation} of the family of nilcharacters $\chi_h$.
We turn to the details of this now.

\section{Linearisation}\label{approx-lin-sec}

We now take the approximate linearity relationship in Corollary \ref{approxa} and see what this implies about the family of nilcharacters $\chi_h$.  As motivation, we begin by discussing the $s=2$ case, which was treated in \cite{gowers-4aps} and developed further in \cite{green-tao-u3inverse}. Here, the one-step nilcharacters $\chi_h$ take the form $\chi_h(n) = e(\xi_h n)$ for some frequency $\xi_h \in \R/\Z$. Corollary \ref{approxa} asserts in this case that the map $h \mapsto \xi_h$ is approximately linear
in the sense that
\begin{equation}\label{xha}
 \xi_{h_1} + \xi_{h_2} - \xi_{h_3} - \xi_{h_4} = O(\frac{1}{N}) \hbox{ mod } 1
\end{equation}
for many additive quadruples $h_1+h_2=h_3+h_4$.

This type of constraint was analysed in \cite{gowers-4aps}, using what is now called the \emph{Balog-Szemer\'edi-Gowers lemma} \cite{balog,gowers-4aps}, together with a version of Freiman's inverse sumset theorem \cite{freiman} due to Ruzsa \cite{ruzsa-freiman}.  As a consequence of these tools from additive combinatorics and a little extra \emph{geometry of numbers}, one can deduce from \eqref{xha} that the map $h \mapsto \xi_h$ is somewhat \emph{bracket-linear}, in the sense that there exist real numbers $\alpha_1,\ldots,\alpha_m,\beta_1,\ldots,\beta_m,\gamma$ for some $m=O(1)$ such that one has the relation
\begin{equation}\label{xih-linear}
 \xi_h = \sum_{j=1}^m \alpha_j \lfloor \beta_j h \rfloor + \gamma + O(\frac{1}{N}) \hbox{ mod } 1
\end{equation}
for many values of $h$. See \cite{green-tao-u3inverse} for further discussion and \cite[Appendix C]{u4-inverse} for a guide to how to use the arguments of \cite{green-tao-u3inverse} to supply a proof of this exact claim, which was not required there. In particular, we can approximate $\chi_h(n)$ (modulo ``lower order terms'') by the expression
\begin{equation}\label{chan}
 \chi(h,n) := e(\gamma n) \prod_{j=1}^m e( \alpha_j n \lfloor \beta_j h \rfloor ).
\end{equation}
A new innovation in our longer paper to come is to view \eqref{chan} as a (piecewise) ``bi-nilcharacter'' of two variables $h, n$, which is of ``bi-degree'' $(1,1)$
in $h,n$.  Informally, this means that each bracket monomial that comprises the phase of $\chi(h,n)$ is of degree at most $1$ in $h$ and of degree at most $1$ in $n$.  Properly formalising this notion of bi-degree involves setting up the notion of a polynomial sequence in quite general filtered nilmanifolds; this will be done in the full paper \cite{gtz-inverse} and we shall say little more about it here. Rather, we shall limit ourselves to an illustrative example, namely that of describing the sequence $(h,n) \mapsto e( \alpha n \lfloor \beta h \rfloor )$ as a (piecewise) bi-nilcharacter of bi-degree $(1,1)$.  

By almost exactly the same computation as in \S \ref{sec3} we see that 
\[ e(\alpha n \lfloor \beta h \rfloor) = F(g(h,n)\Gamma),\] where here we are working on the Heisenberg nilmanifold $G/\Gamma$, the function $F$ is given by $F(x,y,z) = e(-z)$ as before, and now
\[ g(h,n) := e_1^{\alpha n} e_2^{\beta h}.\]

Once again we must note that $F$ is not Lipschitz, but we shall imagine that it is for the purposes of this discussion. Given this, the key feature that qualifies $e(\alpha n \lfloor \beta h \rfloor)$ as a bi-nilcharacter is that the polynomial sequence $g(h,n)$ has bi-degree $(1,1)$ in the variables $h,n$.  What does this mean? If one introduces the partial derivative operators
$$ \partial_h^a g(h,n) := g(h+a,n) g(h,n)^{-1}$$
and
$$ \partial_k^b g(h,n) := g(h,n+b) g(h,n)^{-1},$$
then we can easily verify that $\partial_h^a \partial_h^b g$ and $\partial_k^a \partial_k^b g$ are trivial, that $\partial_h^a \partial_k^b g$ or $\partial_k^b \partial_h^a b$ takes values in $G_2 = [G,G]$, and that any triple derivative of $g$ is trivial.  It is this package of properties that we refer to as being of bi-degree $(1,1)$ in the $h,n$ variables.  More generally, to define a bi-nilsequence of bi-degree $(p,q)$, one needs to endow the nilpotent group $G$ with a two-parameter filtration $(G_{(i,j)})_{i,j \geq 0}$ obeying the inclusions $G_{(i,j)} \supseteq G_{(i',j')}$ when $i' \geq i, j' \geq j$ and $[G_{(i,j)},G_{(k,l)}] \subseteq G_{(i+k,j+l)}$ for $i,j,k,l \geq 0$, and ask that the sequence $g(h,n)$ be such that any mixed derivative involving $i$ differentiations in the $h$ variable and $j$ differentiations in the $n$ variable takes values in $G_{(i,j)}$.  See \cite{gtz-inverse} for details.

An example to keep in mind for a bi-nilcharacter of bi-degree $(p,q)$ is that of a polynomial phase
\begin{equation}\label{polyhn}
 (h,n) \mapsto e( \sum_{i=0}^p \sum_{j=0}^q \alpha_{i,j} h^i n^j ).
\end{equation}
As with our earlier discussion of 1-variable nilsequences this is not an especially representative example and one also needs to model ``bracket polynomial'' behaviour.  To give a more complicated example than the one just discussed arising from the Heisenberg nilmanifold, $$ e( \alpha n \lfloor \beta h \lfloor \gamma n \rfloor \rfloor )$$ is a (piecewise) bi-nilcharacter of bi-degree $(1,2)$ in $h,n$.

Now we turn to higher step analogues of the phenomena just discussed. 

\begin{theorem}[Linearisation]\label{linear}  Suppose that $f: [N] \to \D$ is a function such that for many $h$ in $[-N,N]$ the multiplicative derivative $\Delta_h f$ correlates with an $(s-1)$-step nilcharacter $\chi_h$.  Then there exists a bi-nilcharacter $\chi(h,n)$ of bi-degree $(1,s-1)$ in $h,n$ and $(s-2)$-step nilsequences $\psi_h$ such that $\Delta_h f$ correlates with $\chi(h,\cdot) \psi_h$ for many $h \in [-N,N]$.\end{theorem}

\emph{Remark.} Note that in the case $s = 1$ this is more-or-less precisely the outcome of the discussion we had above, in which the phase $\xi_h$ was shown to vary bracket-linearly and then exhibited as a  bi-nilcharacter coming from the Heisenberg group.

We refer to this operation of replacing the family of one-dimensional $(s-1)$-step nilcharacters $\chi_h(n)$ by a single ``bi-nilcharacter'' $\chi(h,n)$ of degree $(1,s-1)$ in $h,n$ as \emph{linearisation}.  Establishing this property is difficult, and occupies the bulk of \cite{gtz-inverse}.  The starting point for accomplishing this linearisation will be the top-degree portion of the approximate cocycle equation, Lemma \ref{gow}, or in other words Corollary \ref{approxa}.  In the converse direction, it is not difficult to show by an algebraic computation that if $\chi(h,n)$ is a bi-nilcharacter of bi-degree $(1,s-1)$ in $h,n$, then the one-dimensional nilcharacters $\chi_h(n) := \chi(h,n)$ obeys the conclusion of Corollary \ref{approxa}. The reader is invited to do this for the simple example of the polynomial phase \eqref{polyhn} with $(p,q)=(1,s-1)$.

To obtain linearisation from Corollary \ref{approxa} for a general value of $s \geq 3$ requires five additional ingredients.
\begin{enumerate}
\item A secondary induction on the ``rank'' of the nilcharacters being linearised.
\item A ``sunflower decomposition'' that regularises the frequencies involved into ``petal'' and ``core'' frequencies. Roughly speaking, the core frequencies do not depend on $h$ whilst the petal frequencies vary in a highly independent fashion with $h$. 
\item A ``Furstenberg-Weiss argument'', based ultimately on the quantitative equidistribution theory of nilsequences, that shows that every top order term in a nilcharacter has at most one petal (genuinely $h$-dependent) frequency. 
\item A further application of the quantitative equidistribution theory of nilsequences, together with additive combinatorics, to show that these petal frequencies (may be assumed to) vary bracket-linearly. 
\item An algebraic construction to model these objects, which vary bracket-linearly in $h$ and in a ``nil-fashion'' on $n$, by a bi-nilsequence $\chi(h,n)$ of bi-degree $(1,s-1)$.
\end{enumerate}

We now give a few further details for each of these (somewhat technical) ingredients in turn.\vspace{8pt}

(i) \emph{The notion of degree and rank.} The need for an induction on rank first arose in the $s=2$ case of linearisation in \cite{u4-inverse}, in which the (piecewise) nilcharacters $\chi_h$ took the form
$$ \chi_h(n) = e( \sum_{i=1}^{m_h} \alpha_{h,i} n \lfloor \beta_{h,i} n \rfloor + \gamma_{h,i} n^2 + \ldots ),$$
where the $\ldots$ denote $1$-step factors.  It turned out that one had to first fully linearise the ``rank $2$ quadratics'' $\alpha_{h,i} n \lfloor \beta_{h,i} n \rfloor$ before one could then linearise the ``rank $1$ quadratics'' $\gamma_{h,i} n^2$, because the process of linearising the former type of quadratic tended to generate error terms that would have to be absorbed into the latter type of quadratic.  A typical example of such a manipulation arises from the identity
\begin{equation}\label{alphab}
 e( \alpha n \lfloor \beta n \rfloor ) = e( - \beta n \lfloor \alpha n \rfloor ) e( \alpha \beta n^2 ) e( \{ \alpha n \} \{ \beta n \} )
 \end{equation}
which equates the rank $2$ quadratic $e( \alpha n \lfloor \beta n \rfloor )$ with the rank $2$ quadratic $e( -\beta n \lfloor \alpha n \rfloor )$ modulo rank $1$ quadratic and $1$-step errors.

In the higher step case, one would like to similarly organise various components of an $(s-1)$-step nilcharacter into components of different ranks.  If one pretends that a nilcharacter $\chi$ is built up of various bracket monomials of degree $(s-1)$, times lower order terms, then one can heuristically think of the rank of each monomial as the number of brackets involved in its definition, plus one. For instance, $e( \alpha n \lfloor \beta n \lfloor \gamma n^2 \lfloor \delta n \rfloor \rfloor \rfloor )$ is a degree $5$ bracket monomial with a rank of $4$.

One can formalise the notion of rank using the calculus of bracket polynomials, but the approach taken in \cite{gtz-inverse} is to abstract away the bracket polynomials and define rank purely within the formalism of nilcharacters.  This is done by a device similar (though not identical) to that used to define bi-nilcharacters of a given bi-degree.  Namely, to build an $(s-1)$-step nilcharacter $\chi$ of a given rank $r_0$, one creates a two-dimensional filtration $G_{(d,r)}$ on a nilpotent group $G$ for every given degree $d$ and rank $r$, with the nesting properties $G_{(d,r)} \supseteq G_{(d',r')}$ when $d' > d$ or $d'=d$ and $r'>r$, as well as the inclusions $[G_{(d,r)}, G_{(d',r')}] \subseteq G_{(d+d',r+r')}$ for all $d, r \geq 0$ and $G_{(d,0)} = G_{(d,1)}$, with the hypothesis that $G_{(s-1,r_0+1)}$ vanishes.  One then writes $\chi(n) = F(g(n)\Gamma)$ where $F$ obeys suitable Lipschitz and vertical character properties, and $g$ is a polynomial sequence with the property that the $i$-fold derivatives take values in $G_{(i,0)} = G_{(i,1)}$ for all $i \geq 0$.  For details, see \cite{gtz-inverse}.\vspace{8pt}

(ii) \emph{The sunflower decomposition}.  Suppose that we are dealing with the case $s = 3$ and that, for the sake of exposition, we have $\chi_h(n) = e(\alpha_h n \lfloor \beta_h n\rfloor)$. At this stage we have no information about how the frequencies $\alpha_h, \beta_h$ vary with $h$. It may be that $\alpha_h$ is roughly constant in $h$ and that $\beta_h$ is highly oscillatory in $h$. If this is the case we are actually quite happy, since then some understanding of the distribution of $\chi_{h_1}\chi_{h_2}\overline{\chi_{h_3}\chi_{h_4}}(n)$ as $h_1,h_2,h_3,h_4$ vary over additive quadruples is possible. More bothersome is the possibility of behaviour that is a mix of these two extremes, and the sunflower decomposition exists to rule this out. 

Suppose that in some more general setting the set of frequencies of $\chi_h$ is some set $\Xi_h$ of size $O(1)$. In the example just described we have $\Xi_h = \{\alpha_h, \beta_h\}$ but in higher-step settings these frequencies might come from a host of bracket expressions such as $e(\alpha_h n \lfloor \beta_h n\lfloor \gamma_h n\rfloor \rfloor)$ or $e(\alpha_h n \lfloor \beta_h n\rfloor \lfloor \gamma_h n\rfloor)$ or the product of several such terms. The aim of the sunflower decomposition is to replace the sets $\Xi_h$ by new sets
\begin{equation}\label{decomp} \tilde \Xi_h = \Xi_* \cup \Xi'_h,\end{equation} all these sets still having size $O(1)$. Every frequency in $\Xi_h$ is an $O(1)$-rational combination of those in $\tilde \Xi_h$, up to a small error. The ``core'' set $\Xi_*$ consists of frequencies which do not depend on $h$, whilst the ``petal'' sets $\Xi'_h$ depend on $h$ in a very dissociated manner: for most triples $h_1, h_2, h_3$ the frequencies in the union $\Xi_* \cup \Xi'_{h_1} \cup  \Xi'_{h_2}\cup  \Xi'_{h_3}$ do not approximately satisfy an $O(1)$-rational relation. 

We shall say nothing about the proof of the sunflower decomposition here, other than that it may be established by iterative refinement; if at some stage the requirements are not met by \eqref{decomp}, it is possible to add a new frequency to the core set and reduce the size of many of the petal sets $\Xi'_h$. Slightly implicitly, this argument may be read out of \cite[Section 7]{u4-inverse}, particularly Proposition 7.5. 

Once the sunflower decomposition has been established some work is required to express the original nilcharacter $\chi_h(n)$ in terms of objects involving the new sets of frequencies $\tilde \Xi_h$. Recall that the original frequencies $\Xi_h$ are $O(1)$-rational combinations of the $\tilde \Xi_h$, up to $O(1)$. In our work on $\GI(3)$ this was done explicitly using ``bracket quadratic identities'', the basic idea being that an object such as $e(\alpha n_1 \lfloor \beta n_2\rfloor)$ is multilinear up to lower-order terms. In the more general paper to come, these issues are instead dealt with in a more abstract fashion, using nilsequences.\vspace{8pt}

(iii) \emph{The Furstenberg-Weiss argument.} For simplicity let us suppose that $s = 3$ and imagine that, following step (ii), the top-order term of $\chi_h(n)$ is a product of terms such as $e(\alpha_h n \lfloor \beta_h n\rfloor)$, where the frequencies $\alpha_h, \beta_h$ belong to frequency sets $\Xi_h$ which have been decomposed as $\Xi_* \cup \Xi'_h$ according to the sunflower decomposition. The aim is to show that (after refining the set of $h$) we do not have $\alpha_h, \beta_h \in \Xi'_h$. That is to say, there are no terms with more than one petal frequency. Put another way, no more than one frequency  in any bracket monomial genuinely depends on $h$. 

The argument proceeds by studying the conclusion of Corollary \ref{approxa} using an argument of Furstenberg and Weiss.  For simplicity, let us just discuss a model case in which $s=3$ and each $\chi_h$ is essentially of the form $\chi_h(n) = e( \alpha_h n \lfloor \beta_h n \rfloor )$. This was already treated in detail in \cite[Lemma 7.3]{u4-inverse}.

\begin{lemma}[Furstenberg-Weiss argument, model case]  Suppose that for $ijk=123,124$, the six frequencies $\alpha_{h_i}, \beta_{h_i}, \alpha_{h_j}, \beta_{h_j}, \alpha_{h_k}, \beta_{h_k}$ are linearly independent in the sense that there is no non-trivial linear combination of these six frequencies with bounded integer coefficients that is equal to $O(1/N)$ modulo $1$. Then $\chi_{h_1} \chi_{h_2} \overline{\chi_{h_3}} \overline{\chi_{h_4}}$ has negligible mean, and more generally does not correlate with any $1$-step nilsequence.
\end{lemma}
We remark that we will find ourselves in exactly this situation if there are many $h$ such that $\chi_h(n)$ contains a petal-petal combination. The conclusion of this lemma then contradicts Corollary \ref{approxa}.
\begin{proof}(Sketch)  For notational simplicity we just sketch the claim that the mean 
\begin{equation}\label{enn}
\E_{n \in [N]} \chi_{h_1} \chi_{h_2} \overline{\chi_{h_3}} \overline{\chi_{h_4}}(n)
\end{equation}
is negligible.  We write each $\chi_{h_j}(n)$ as a nilcharacter
$$ \chi_{h_j}(n) = F_j( e_{j,1}^{\alpha_{h_j} n} e_{j,2}^{\beta_{h_j} n} \hbox{ mod } \Gamma_j)$$
where $e_{j,1}, e_{j,2}$ generate copies $G_j$ of the Heisenberg group with corresponding discrete subgroups $\Gamma_j$, and $F_j$ is a suitable function.  The mean \eqref{enn} is then controlled by the equidistribution of the orbit
$$ (e_{j,1}^{\alpha_{h_j} n} e_{j,2}^{\beta_{h_j} n} \hbox{ mod } \Gamma_j)_{j=1}^4$$
in a product $(G_1/\Gamma_1) \times \ldots \times (G_4/\Gamma_4)$ of four Heisenberg nilmanifolds.

An application of a quantitative version of Leibman's theorem \cite{leibman-poly} on equidistribution of polynomial orbits in nilmanifolds, established by the first two authors in \cite{green-tao-nilratner}, tells us (roughly speaking) that this orbit is equidistributed on a subnilmanifold $H/\Sigma$ of $(G_1/\Gamma_1) \times \ldots \times (G_4/\Gamma_4)$, where $H$ is a closed subgroup of $G_1 \times \ldots \times G_4$; the mean \eqref{enn} is then essentially the integral of the tensor product $F_1 \otimes F_2 \otimes \overline{F_3} \otimes \overline{F_4}$ on this subnilmanifold.  The linear independence of the frequencies $\alpha_{h_i}, \beta_{h_i}, \alpha_{h_j}, \beta_{h_j}, \alpha_{h_k}, \beta_{h_k}$ for $ijk=123$ can be used to show that the projection from $H$ to $G_1 \times G_2 \times G_3$ is surjective; similarly, the same hypothesis for $ijk=124$ can be used to show that the projection from $H$ to $G_1 \times G_2 \times G_4$ is surjective.  Taking commutators, one then concludes that $H$ contains $[G_1,G_1] \times \{\id\} \times \{\id\} \times \{\id\}$ as a subgroup.  From this and the non-trivial oscillation of $F_1$ we see that $F_1 \otimes F_2 \otimes \overline{F_3} \otimes \overline{F_4}$ has mean zero on $H/\Sigma$, and the claim follows.
\end{proof}
The above argument may be used to rule out the possibility that $\chi_h(n) = e(\alpha_h n \lfloor \beta_h n\rfloor)$ with both $\alpha_h$ and $\beta_h$ being petal frequencies, since in this case almost all additive quadruples $h_1+ h_2 = h_3 + h_4$ will satisfy the hypotheses of the lemma, leading to a contradiction of Corollary \ref{approxa}. A very similar, but more notationally intensive, argument may be used to rule out a more general possibility: that $\chi_h(n)$, which could in general be a product of \emph{many} terms like $e(\alpha_h n\lfloor\beta_h n\rfloor)$, contains one such term with two petal frequencies.\vspace{8pt}

(iv) \emph{Additive Combinatorics.} Let us persist with the model setting in which $s = 3$ and $\chi_h(n)$ is a product of terms of the form $e(\alpha_h n\lfloor\beta_h n\rfloor)$. As a consequence of part (iii), we may assume that in each such term only one of $\alpha_h, \beta_h$ genuinely depends on $h$ (i.e. is a petal frequency), the other frequency being core. A simple model to consider is that in which $\chi_h(n) = e(\alpha_h n\lfloor\beta n\rfloor)$. 

We then re-examine Corollary \ref{approxa} in the light of this new structural information on $\chi_h(n)$. By a further argument of Furstenberg-Weiss type, very similar to the above, one may show that $\alpha_h$ satisfies a relation of type \eqref{xha}. Applying the same additive-combinatorial machinery (the Balog-Szemer\'edi-Gowers theorem and Fre\u{\i}man's theorem) we may replace $\alpha_h$ by a bracket-linear object as in \eqref{xih-linear}. Details of this type of argument in the case of $\GI(3)$ may be found in \cite[Section 8]{u4-inverse}.\vspace{8pt}

(v) \emph{Constructing a nilobject.} We have, at this point, shown that the top-order terms of $\chi_h(n)$ vary in a somewhat ``rigid'' or algebraic way -- more specifically, the $h$-dependence is bracket-linear. The remaining task in the ``linearisation'' part of the argument is to identify these top-order terms as coming from a bi-nilsequence $\chi(h,n)$. In previous works on the inverse conjectures such as that of the first two authors on the $U^3$-norm \cite{green-tao-u3inverse} and the authors' treatment of the $U^4$-norm \cite{u4-inverse} this ``nilobject'' was constructed in a rather \emph{ad hoc} manner. In the former paper suitable products of Heisenberg nilmanifolds were exhibited, whilst in the latter the free $3$-step nilpotent group on a suitable number of generators was considered. We also remark that, in both of these works, the nilobject was constructed at the very last step of the argument, rather than prior to the symmetry argument (to be discussed in the next section) as here. In our longer paper \cite{gtz-inverse} we introduce a more systematic construction based on a semidirect product. Rather than describe this in any kind of generality we merely outline an example of the construction. Suppose that $\alpha_h := \gamma \{\delta h\}$ and that we know, for fixed $h$, how to construct the nilcharacter $\chi_h(n) =  e(\alpha_h n\lfloor\beta n\rfloor)$. We do, of course, since it comes from a Heisenberg example: however the description that follows works in much greater generality. Then we show how to realise $\chi_h(n)$ as a bi-nilsequence.

The reader might briefly recall, at this point, the construction of $\chi_h(n)$ as a nilcharacter on the Heisenberg as given in \S \ref{sec3}, namely
\[ \chi_h(n) = F(g_h(n)\Gamma)\] with $g_h(n) = e_1^{\alpha_h n} e_2^{\beta n}$. We note once more that $F$ is not Lipschitz, and so $\chi_h(n)$ is not quite a true nilcharacter, but we shall pretend that it is for the purposes of this announcement. 
\newcommand\petal{\operatorname{petal}}

We turn now to the interpretation of $\chi_h(n)$ as a bi-nilsequence in $h$ and $n$. 
The first task is to identify a subgroup $G_{\petal}$ of the Heisenberg group $G$ representing that part of $G$ that is ``influenced by'' the petal frequency $\alpha_h$. In our setting this is very easy; simply take $G_{\petal}$
to be the subgroup of $G$ generated by $e_1$ and $[e_1, e_2]$.
Note that $G_{\petal}$ is abelian and normal in $G$. These features are quite general and hinge on the fact that there is only \emph{one} petal frequency in $\chi_h(n)$. Of course, it was precisely to achieve this that we worked so hard in (iii) above. In particular $G$ acts on $G_{\petal}$ by conjugation and we may form the semidirect product $G \ltimes G_{\petal}$, defining multiplication by
\[ (g, g_1)\cdot (g', g'_1) = (gg', g_1^{g'} g'_1),\] where $a^b := b^{-1} a b$ denotes conjugation.

Now consider the action $\rho$ of $\R$ on $G \ltimes G_{\petal}$ defined by 
\[ \rho(t)(g, g_1) := (g g_1^t, g_1).\] We may form a further semidirect product
\[ \tilde G := \R \ltimes_{\rho} (G \ltimes G_{\petal}),\] in which the product operation is defined by
\[ (t, (g, g_1)) \cdot (t', (g', g'_1)) = (t + t', \rho(t')(g, g_1) \cdot (g', g'_1)).\]
$\tilde G$ is a Lie group; indeed, one easily verifies that it is $3$-step nilpotent. Inside $\tilde G$ we take the lattice 
\[ \tilde \Gamma := \Z \ltimes_{\rho} (\Gamma \ltimes \Gamma_{\petal}),\] where $\Gamma_{\petal} := \Gamma \cap G_{\petal}$.

We will construct $\chi_h(n)$ as a bi-nilsequence $\tilde F(\tilde g(h,n) \tilde \Gamma)$ for suitable $\tilde F : \tilde G/\tilde\Gamma \rightarrow \C$ and an appropriate polynomial sequence $\tilde g : \Z^2 \rightarrow \tilde G$. For $\tilde g$, take
\[ \tilde g(h, n) := (0, (e_2^{\beta n}, e_1^{\gamma n})) \cdot (\delta h, (\id, \id))\]
and observe that
\begin{align*} \tilde g(h,n)\tilde\Gamma & = (0, (e_2^{\beta n}, e_1^{\gamma n})) \cdot (\{\delta h\} , (\id, \id)) \tilde \Gamma \\ & = (\{\delta h\}, (e_2^{\beta n} e_1^{\{\delta h\}\gamma n}, e_1^{\gamma n}))\tilde \Gamma.
\end{align*}
Finally, take $\tilde F : \tilde G/\tilde \Gamma \rightarrow \C$ to be the function defined by
\[ \tilde F((t, (g, g'))\tilde\Gamma) = F(g)\] whenever $0 \leq t < 1$ and $g$ lies in the fundamental domain of $G/\Gamma$.  By exactly the same computation as for the Heisenberg group we have 
\[ \tilde F(\tilde g(h,n)\tilde\Gamma) = e(\gamma\{\delta h\}  n \lfloor\beta n\rfloor) = \chi_h(n),\] which is exactly what we wanted.

This completes the discussion of point (v) in the model case of a rather clean and simple collection of nilcharacters $\chi_h(n)$ on the Heisenberg group. Even here, we have omitted details: for example, one must carefully place a filtration on $\tilde G$ and confirm that the new bi-nilsequence $\chi(h,n)$ has the claimed bi-degree, namely $(1,2)$ in this case (note, however, that we have not even properly \emph{defined} bi-degree in this announcement). The difficulties involved in doing this, and in generalising the semidirect product construction just described, are largely notational.\vspace{8pt}

With a brief discussion of each of the five points (i) to (v) now completed, we have concluded our sketch proof of Theorem \ref{linear}.

\section{Symmetrisation}

We turn now to the final part of the argument. Let us begin with a summary of our current position, which is the result of applying the observation \eqref{enhf-char} and the rather substantial Theorem \ref{linear}. Together, these tell us that if $f: [N] \to \D$ is a function with large $U^{s+1}[N]$-norm then there is a bi-nilcharacter $\chi(h,n)$ of bi-degree $(1,s-1)$ in $h, n$ such that for many $h \in [-N,N]$, $\Delta_h f$ correlates with $\chi(h,\cdot)$ modulo $(s-2)$-step errors.  We would like to ``integrate'' $\chi(h,n)$ by expressing it in the form
$$ \chi(h,n) = \Delta_h \theta(n)  \cdot\psi_h(n),$$
for some $s$-step nilcharacter $\theta$ and some $(s-2)$-step nilcharacters $\psi_h$.  

To see what is necessary to achieve this, let us proceed heuristically as at the start of \S \ref{approx-lin-section} and suppose that $\chi(h,n) = \Delta_h \theta(n)$. Then we have
\[  \chi(h, n+k)\overline{\chi(h,n)} = \Delta_k \Delta_h \theta(n) =\Delta_h \Delta_k \theta(n) = \chi(k, n+h)\overline{\chi(k,n)}.\] This ``symmetry'' relation, which will certainly not be satisfied by an arbitrary binilcharacter $\chi(h,n)$, suggests that, even in our rather weaker setting, we must obtain further information about $\chi$ before we can complete our task.

For instance, if one had 
$$ \chi(h,n) \approx e( \alpha h \lfloor \beta n \rfloor )$$
(where $\approx$ informally denotes equivalence up to lower order terms)
then there does not appear to be any reasonable candidate for the antiderivative $\theta$, whereas if
\[ \chi(h,n) \approx e(\alpha h \lfloor \beta n\rfloor + \alpha n \lfloor \beta h \rfloor)\] then $\chi(h,n) = \Delta_h \theta(n)$ up to lower order terms, where $\theta(n) := e(\alpha n \lfloor \beta n\rfloor)$.
  The obstruction here is analogous to the basic fact in de Rham cohomology that in order for a $1$-form $\omega$ to be exact (i.e. to be the derivative $\omega = df$ of a scalar function), it is first necessary that it be closed (i.e. $d\omega = 0$). 

The need for this symmetry, and the means for obtaining it, was first addressed in \cite{green-tao-u3inverse,sam} as part of the proof of $\GI(2)$. A somewhat different argument of this nature later appeared in \cite{u4-inverse} as part of the proof of $\GI(3)$.  In the former case, this symmetry was obtained by a Cauchy-Schwarz argument that was similar (but subtly different) from the one used to establish Lemma \ref{gow}.  In the latter case, we inspected the lower order terms of Lemma \ref{gow} and we do the same here. In our present setting, this lemma implies that
\begin{equation}\label{correlate-2var} \E_{n \in [N]} \chi(h_1,n)\chi(h_2, n + h_1 - h_4) \chi(h_3, n)\chi(h_4, n + h_1 - h_4) \J_{\lo}(h_i, n) \gg 1\end{equation} for many additive quadruples $h_1 + h_2 = h_3 + h_4$.
Here, and in everything that follows, we use the symbol $\J()$ to denote ``junk terms''. Here, these are terms of ``lower order'' (hence the subscript $\lo$); later on $\J$ will also be allowed to include terms, depending only on some strict subset of the variables, that are destined to be annihilated by applications of the Cauchy-Schwarz inequality. We will denote these by a subscript $\cs$.

Let us pause to recall the remarks immediately following the statement of Lemma \ref{gow} to the effect that very little was ``lost'' in proving that lemma. It should not, therefore, come as a surprise that \eqref{correlate-2var} is \emph{in principle} enough to proceed; however, actually making use of this observation is surprisingly tricky.

Let us specialise to the case $s = 4$ and for the sake of this discussion suppose that $\chi(h,n) = e(T(h,n,n,n))$, where $T : [N]^4 \rightarrow \R/\Z$ is to be thought of as a ``bracket linear form'' such as 
\[ T(n_1, n_2, n_3,n_4) = \alpha n_1 \lfloor \beta n_2 \lfloor \gamma n_3 \lfloor\delta n_4\rfloor \rfloor \rfloor.\] 

Since $T$ only appears in the expression $T(h,n,n,n)$ we may assume that $T$ is already symmetric in the last three variables, by replacing $T(n_1, n_2, n_3,n_4)$ with 
\[ \frac{1}{6}\sum_{\pi \in S_3} T(n_1, n_{\pi(2)}, n_{\pi(3)}, n_{\pi(4)}).\] We need only establish, then, some symmetry in the first two variables of $T$.

Substituting into \eqref{correlate-2var} and parametrising additive quadruples as $h_1 = h$, $h_2 = h+ a + b$, $h_3 = h+ a$, $h_4 = h + b$ we obtain
\begin{align*} \E_{n,h,a,b} e ( & T(h,n,n,n) + T(h + a + b, n - b, n - b,n-b) - \\ & T(h + a, n,n,n) - T(h + b, n - b, n - b,n-b)) \J_{\lo}(\cdot ) \gg 1.\end{align*}	
If $T$ were genuinely quartilinear this would collapse (using the symmetry in the last three variables) to give 
\begin{equation}\label{nab-bias} \E_{n,a,b}e(- 3 T(a,b,n,n)) \J_{\lo}(\cdot )\gg 1,\end{equation} where $\J_{\lo}()$ is only \emph{linear} in $n$. Of course, $T$ is \emph{not} genuinely quartilinear but rather ``bracket quartilinear''. In practice this means that $T$ is quartilinear ``up to lower order terms'', a phenomenon best understood, but perhaps harder to explain in a brief overview, by thinking of $\chi(h,n)$ as a nilobject rather than as a bracket object.  After formalising this approximate quartilinearity, one may eventually assert, in place of \eqref{nab-bias}, a statement of the form
\begin{equation}  \E_{n,a,b} e(-3 T(a,b,n,n)) \J_{\lo, \cs}( \cdot) \gg 1,\end{equation}
wherethe subscript $\cs$ in $\J_{\lo, \cs}( \cdot)$ implies that this error term is not necessarily of lower order (degree $1$) in $n$, but the non-linear terms depend only on one of the variables $a,b$ and will at some later point be removed using the Cauchy-Schwarz inequality.
Applying Cauchy-Schwarz to this yields
\[ \E_{n,a,b,b'} e(-3T(a,b,n,n) + 3 T(a,b',n,n)) \J_{\lo, \cs} (\cdot ) \gg 1.\]
Now the non-linear terms in  $\J_{\lo, \cs}(\cdot)$  are independent of $a$, and depend only on one of the variables $b,b'$.
Substituting $c:= a + b + b'$ gives
\[ \E_{n,c,b,b'} e(-3T(c - b - b', b, n,n) + 3 T(c - b - b', b', n,n)) \J_{\lo, \cs}(\cdot ) \gg 1.\] 
In particular, there is some value of $c$ such that
\[ \E_{n,b,b'} e(-3T(c - b - b', b, n,n) + 3 T(c - b - b', b', n,n)) \J_{\lo, \cs}(\cdot ) \gg 1.\]
Using multilinearity (modulo lower order terms) and absorbing any terms depending on only one of $b,b'$ into the junk term $\J( )$ we obtain
\[ \E_{n,b,b'} e(3T(b',b, n,n) - 3T(b,b',n,n)) \J_{\lo, \cs}( \cdot) \gg 1.\] 
At this point we have a statement that certainly seems to be asserting at least some kind of symmetry in the first two variables of $T$, which is of course our eventual goal. 

Further manipulations are required to turn it into something usable. Write 
\[ \psi(b,b',n,n) := 3T(b',b,n,n) - 3T(b,b',n,n);\]
thus
\begin{equation}\label{eq71} \E_{n,b,b'} e(\psi(b,b',n,n)) \J_{\lo,\cs}(\cdot) \gg 1.\end{equation}
The junk term $\J(\cdot)$ is comprised of terms $\J_{\lo}$ of lower order in $n$, and also of terms $\J_{\cs}$ depending on only one of the variables $b,b'$. By two applications of the Cauchy-Schwarz inequality we may eliminate these latter terms, obtaining
\begin{align*} \E_{n,b_1,b'_1, b_2, b'_2}  e(\psi(b_1, b'_1,n,n) - &\psi(b_2, b'_1,n,n) - \\ & \psi(b_1, b'_2,n,n) + \psi(b'_1, b'_2, n,n)) \J_{\lo}(\cdot ) \gg 1,\end{align*} 
where now the junk term $\J_{\lo}$ consists only of terms that are of linear nature in $n$. In particular, on average in $b_1,b_1',b_2,b_2'$, the Gowers $U^2$-norm of 
\[
 e(\psi(b_1, b'_1,n,n) - \psi(b_2, b'_1,n,n) -  \psi(b_1, b'_2,n,n) + \psi(b'_1, b'_2, n,n))
\]
is large. Writing this out in full and using the fact that $\psi$ is quartilinear up to lower order terms implies that
\[
e(2\psi(b_1, b_1', h_1, h_2) -2\psi(b_2, b_1', h_1, h_2) - 2\psi(b_1, b_2', h_1, h_2) +2\psi(b_1, b_2', h_1, h_2)) 
\]
correlates with a lower-order object.
By pigeonhole there is some choice of $b_2, b'_2$ such that the expectation over the remaining variables $h_1,h_2, b_1, b'_1$ is still $\gg 1$. For these fixed $b_2,b'_2$ the terms involving $b_2, b'_2$ are of lower order in $b_1, b'_1, h_1,h_2$, as a result of which we conclude that
\[ e(2\psi(b, b', h_1, h_2)) = e(6 T(b',b,h_1, h_2) - 6T(b,b',h_1, h_2))\] 
correlates with a lower-order object.

This expresses a certain symmetry of $T(n_1,n_2,n_3,n_4)$ in the first two variables, and this is enough to complete the ``integration'' of $\chi(h,n)$ and hence the proof of the inverse conjecture $\GI(s)$.

\providecommand{\bysame}{\leavevmode\hbox to3em{\hrulefill}\thinspace}

\end{document}